\documentclass[11pt,reqno]{amsart}

\usepackage{amsmath, amsthm, amssymb}
\usepackage{amsfonts}
\usepackage[ansinew]{inputenc}
\usepackage[dvips]{epsfig}
\usepackage{graphicx}
\usepackage[english]{babel}
\usepackage{comment}
\usepackage{graphicx}
\usepackage{hyperref}
\usepackage{bbm} % for the characteristic function \mathbbm{1}
\usepackage{cite}
\usepackage{graphicx}
\usepackage{amscd}
\usepackage{xcolor}
\usepackage{bm}
\usepackage{enumerate}

\usepackage{verbatim}
\usepackage{hyperref}
\usepackage{amstext}
\usepackage{latexsym}
\let\oldsqrt\sqrt
\def\sqrt{\mathpalette\DHLhksqrt}
\def\DHLhksqrt#1#2{%
	\setbox0=\hbox{$#1\oldsqrt{#2\,}$}\dimen0=\ht0
	\advance\dimen0-0.2\ht0
	\setbox2=\hbox{\vrule height\ht0 depth -\dimen0}%
	{\box0\lower0.4pt\box2}}

\allowdisplaybreaks

\newcommand{\R}{\mathbb{R}} % reelle Zahlen
\renewcommand{\phi}{\varphi}

\newcommand{\cC}{{\mathcal C}}

\newcommand{\cH}{{\mathcal H}}

\newcommand{\cK}{{\mathcal K}}

\newcommand{\cN}{{\mathcal N}}

\theoremstyle{definition}
\newtheorem{defi}{Definition}[section]
\newtheorem{remark}[defi]{Remark}

\theoremstyle{plain} %default%plain
\newtheorem{thm}[defi]{Theorem}
\newtheorem{prop}[defi]{Proposition}
\newtheorem{lemma}[defi]{Lemma}

\theoremstyle{definition}

\numberwithin{equation}{section}

\title[Symmetry breaking and multiplicity for supercritical Hamiltonian systems]{Symmetry breaking and multiplicity for supercritical elliptic Hamiltonian systems in exterior domains} 

\author[Remi Yvant Temgoua]{Remi Yvant Temgoua} 

\address{School of Mathematics and Statistics, Carleton University, Ottawa, Ontario, Canada.}

\email{remiyvanttemgoua@cunet.carleton.ca} 

\date{\today}

\begin{document}
\maketitle

\begin{abstract}
	We consider positive solutions of the following elliptic Hamiltonian systems
\begin{equation}\label{e_0}
\left\{
\begin{aligned}
-\Delta u+u&=a(x)v^{p-1}~~~~~\text{in}~~A_R\\
-\Delta v+v&=b(x)u^{q-1}~~~~~\text{in}~~A_R\\
u, v&>0~~~~~~~~~~~~~~~\text{in}~~A_R\\
u=v&=0~~~~~~~~~~~~~~~\text{on}~~\partial A_R,
\end{aligned}
\right.
\end{equation}
where $A_R=\{x\in\R^N: |x|>R\},~R>0,~N>3$, and $a(x)$ and $b(x)$ are positive continuous functions. Under certain symmetry and monotonicity properties on $a(x)$ and $b(x)$, we prove that \eqref{e_0} has a positive solution for $(p,q)$ above the standard critical hyperbola, that is, $\frac{1}{p}+\frac{1}{q}<1-\frac{2}{N}$, enjoying the same symmetry and monotonicity properties as the weights $a$ and $b$. In the case when $a(x)=b(x)=1$, our result ensures multiplicity as it provides $\Big\lfloor\frac{N}{2}\Big\rfloor-1$ (being $\lfloor\frac{N}{2}\rfloor$ the floor of $\frac{N}{2}$) non-radial positive solutions provided that
 \begin{equation}
        (p-1)(q-1)>\Big(1+\frac{2N}{\Lambda_H}\Big)^2\Big(\frac{q}{p}\Big),
    \end{equation}
where $\Lambda_H$ is the optimal constant in Hardy inequality for the domain $A_R$.
	
	%Our result provides for the first time a positive answer to the conjecture given in \cite[Page 4]{fall2022existence}, in the case of a large annulus. 
\end{abstract}
{\footnotesize
	\begin{center}
	\end{center}
}

\section{Introduction and main results}\label{section:introduction}
In this paper, we study symmetry breaking and multiplicity of positive solutions to the problem 
\begin{equation}\label{e-1}
\left\{
\begin{aligned}
-\Delta u+u&=a(x)v^{p-1}~~~~~\text{in}~~A_R\\
-\Delta v+v&=b(x)u^{q-1}~~~~~\text{in}~~A_R\\
u, v&>0~~~~~~~~~~~~~~~\text{in}~~A_R\\
u=v&=0~~~~~~~~~~~~~~~\text{on}~~\partial A_R,
\end{aligned}
\right.
\end{equation}
where $A_R=\{x\in\R^N: |x|>R\}$, $q\geq p>2$, and $a, b\in C(\overline{A_R})$ with $a(x)\geq a_0>0$ and $b(x)\geq b_0>0$ for some positive constants $a_0$ and $b_0$. The exponents $p$ and $q$ are above the critical hyperbola, namely
\begin{equation}\label{supercritical-condition}
    \frac{1}{p}+\frac{1}{q}<1-\frac{2}{N}.
\end{equation}
Notice that condition \eqref{supercritical-condition} is the well-known notion of \textit{supercriticality} associated to \eqref{e-1}, see  \cite{bonheure2014hamiltonian,clement1992positive,peletier1992existence}.

In the previous decades, there has been active research on the study of elliptic systems both in bounded and unbounded domains. Elliptic systems of the type \eqref{e-1} in bounded domains have been studied in \cite{clement1992positive,costa1995variational,figueredo1994on,figueredo1996nonquadratic,hulshof1993differential,peletier1992existence,ramos2008solutions}.
%See also \cite{ramos2005spike} where related problem with Neumann boundary condition is considered
The case of the whole space was studied, among other, in \cite{figueredo1998symmetry,serrin1998existence,sirakov2000existence,yang2001nontrivial}. We stress that only subcritical systems have been considered in these papers. For the supercritical case, we quote e.g., \cite{clapp2004strongly,de2003strongly} in which the systems were considered in bounded domains. We also refer to the recent work \cite{kim2021multiple} where critical Hamiltonian systems of the type \eqref{e-1} were considered in bounded domains.  On the other hand, supercritical elliptic Hamiltonian systems without the zero-order term, that is, systems of the form 
\begin{equation}\label{a0}
\left\{
\begin{aligned}
-\Delta u&=a(x)v^{p-1}~~~~~\text{in}~~\Omega\\
-\Delta v&=b(x)u^{q-1}~~~~~\text{in}~~\Omega\\
u, v&>0~~~~~~~~~~~~~~~\text{in}~~\Omega\\
u=v&=0~~~~~~~~~~~~~~~\text{on}~~\partial\Omega,
\end{aligned}
\right.
\end{equation}
where $\Omega=\{x\in\R^N: 0<R_1<|x|<R_2<\infty\}$ is an annulus, have been studied in \cite{abbas2023existence}. See also \cite{cowan2022supercritical1} where \eqref{a0} was considered in arbitrary bounded domains of $\R^N$ in the special case $a(x)=b(x)=1$. 

Our motivation for the study of problem \eqref{e-1} goes back to the very recent work of Boscaggin et al. \cite{boscaggin2023multiplicity} concerning the scalar equation 
\begin{equation}\label{single-equation}
    \left\{
    \begin{aligned}
        &-\Delta u+u=a(x)u^{p-1}~~~\text{in}~~A_R,\\
        &u>0~~~~~~~~~~~~~~~~~~~~~~~~~\text{in}~~~ A_R,\\
        & u\in H^1_0(A_R).
    \end{aligned} 
    \right.
\end{equation}
In \cite{boscaggin2023multiplicity}, the authors used an abstract result due to Cowan and Moameni \cite{cowan2023supercritical} to obtain the existence and multiplicity results of solutions to \eqref{single-equation} for values of $p$ in a suitable range that includes exponents beyond the classical critical Sobolev exponent. Precisely, they considered \eqref{single-equation} along with the corresponding classical energy functional $I: H^1_0(A_R)\cap L^{p}(A_R)\to\R$ given by
\begin{equation*}
    I(u)=\frac{1}{2}\int_{A_R}(|\nabla u|^2+|u|^2)\ dx-\frac{1}{p}\int_{A_R}a|u|^{p}\ dx.
\end{equation*}
They work on a convex cone $\cK$ which consists of positive functions of the form $u=u(r,\theta)$ which are decreasing with respect to $\theta$, that is, $u_{\theta}\leq0$ (being $u_{\theta}$ the weak partial derivative of $u$ with respect to the variable $\theta$). Moreover, they consider $\cN_{\cK}=\{u\in\cK: u\not\equiv0,~I'(u)u=0\}$, the associated Nehari-type set to $\cK$. Then, they develop a mountain pass argument (which relies on variational principle \cite{cowan2023supercritical}) to prove the existence of a solution $u$ of \eqref{single-equation} such that $u\in\cN_{\cK}$ and $I(u)=c_I>0$ where $c_I:=\inf_{u\in\cN_{\cK}}I(u)$ is the Nehari value. One of the crucial points toward this existence result was to prove that the embedding $\cK\hookrightarrow L^p(A_R)$ is compact for all $p>2$. To this end, the authors in \cite{boscaggin2023multiplicity} combine the monotonicity property of functions of $\cK$ with a compact embedding result due to Willem \cite{willem1997minimax} applied in suitably chosen subdomains of $A_R$. In the special case when the weight $a(x)$ is radial, they also obtained a result regarding nonradial solutions under certain assumptions on $R$ or the value of $p$. 

Motivated by the above discussion, we consider in this paper, a class of elliptic systems that generalizes the case of a single equation explored by authors in \cite{boscaggin2023multiplicity}. The main difficulty in studying problem \eqref{e-1} is the lack of compactness due to both the unboundedness of the domain and the supercritical growth of the nonlinearities. However, as in \cite{boscaggin2023multiplicity}, we recover the compactness under the assumption that the weights $a(x)$ and $b(x)$ satisfy certain symmetry and monotonicity properties, working in a closed convex subset $K$ of $W^{2,p'}(A_R)\cap H^1_0(A_R)$-functions (being $p'$ the conjugate of $p$, that is, $p'=\frac{p}{p-1}$) having the same symmetry and monotonicity features as $a(x)$ and $b(x)$. We stress that our analysis to recover the compactness is more delicate due to both the unboundedness of the domain and the fact that we are dealing with systems. It exploits the symmetry and monotonicity properties of functions in $K$ together with Lyon's \cite{lions1982symetrie} compactness result. 

%First, there is the lack of compactness of the Sobolev embedding, since our domain is unbounded and the problem is in a supercritical regime. Second, it is challenging to find an adequate functional, the critical points of which are still solutions of problem \eqref{e-1}, so that we can apply variational methods.

We now wish to state the main results of this paper. To this end, we write $\R^N=\R^m\times\R^n$ with $N=m+n$. We introduce the following variables 
\begin{equation*}
    s=\sqrt{x_1^2+\cdots+x_m^2}~~~~~~~~~~\text{and}~~~~~~~~~~t=\sqrt{x_{m+1}^2+\cdots+x_N^2}.
\end{equation*}
Let $\widehat{A}_R$ be the subset of $\R^2$ defined by
\begin{equation*}
    \widehat{A}_R:=\{(s,t)\in\R^2: s^2+t^2>R^2, s\geq0, t\geq0\}.
\end{equation*}
In polar coordinates, every $(s,t)\in\widehat{A}_R$ can be written as $s=r\cos\theta$ and $t=r\sin\theta$ where $r=\sqrt{s^2+t^2}=\sqrt{x_1^2+\cdots+x_N^2}\in (R,\infty)$ and $\theta\in \Big[0,\frac{\pi}{2}\Big]$ is defined by 
\begin{equation*}
    \theta=\arctan\Bigg(\frac{\sqrt{x_{m+1}^2+\cdots+x_N^2}}{\sqrt{x_1^2+\cdots+x_m^2}}\Bigg)=\arcsin\Big(\frac{1}{r}\sqrt{x_{m+1}^2+\cdots+x_N^2}\Big).
\end{equation*}
In this paper, we work with functions that depend on $r$ and $\theta$ (or equivalently, functions that depend on $s$ and $t$), namely
\begin{equation*}
    u(x)=u(r,\theta)=u(s,t). 
\end{equation*}
Such functions are said to be invariant under the group action $O(m)\times O(n)$, where $O(k)$ is the orthogonal group in $\R^k$.

To state our main results, we assume that the weights $a$ and $b$ satisfy:
\begin{itemize}
    \item [$(\cH):$] $a(x)=a(s,t)$ (resp. $b(x)=b(s,t)$) is a function of $s$ and $t$, and $a(s,t)$ (resp. $b(s,t)$) is a continuously differentiable function with respect to $(s,t)$ and $sa_t-ta_s\leq0$ (resp. $sb_t-tb_s\leq0$) in $\widehat{A}_R$.
\end{itemize}
The first main result of this paper is concerned with the existence in the supercritical regime. It reads as follows. 
\begin{thm}\label{first-main-result}
    Let $3< N=m+n$ with $1< n\leq m$. Let $R>0$ and assume that $a$ and $b$ satisfy $(\cH)$. Let $q\geq p>2$. If 
    \begin{equation}
        \frac{1}{p}+\frac{1}{q}>1-\frac{2}{n+1}=\min\Big\{1-\frac{2}{m+1}, 1-\frac{2}{n+1}\Big\}~~~~\text{for}~~n>\frac{p+1}{p-1},
    \end{equation}
    then problem \eqref{e-1} has a positive weak solution $(u,v)$ invariant under the group action $O(m)\times O(n)$.
\end{thm}
\begin{remark}\label{rmk1}
~
    \begin{itemize}
        \item [$(i)$] Observe that since $m>1$ then $n+1<m+n=N$. Thus,
        \begin{equation*}
            1-\frac{2}{N}>1-\frac{2}{n+1}=\min\Big\{1-\frac{2}{m+1}, 1-\frac{2}{n+1}\Big\}.
        \end{equation*}
        So, our existence result holds for values of $p$ and $q$ beyond the standard critical hyperbola $\frac{1}{p}+\frac{1}{q}=1-\frac{2}{N}$. 
        \item [$(ii)$] The lower bound $n>\frac{p+1}{p-1}$ is a consequence of the compact embedding established in Proposition \ref{compact-embedding}. Moreover, in the case when $n\leq\frac{p+1}{p-1}$, Theorem \ref{first-main-result} holds true without any lower bound condition on $\frac{1}{p}+\frac{1}{q}$. We refer to Remark \ref{rmk2} for more details.
    \end{itemize}
\end{remark}
To prove Theorem \ref{first-main-result}, we utilize, on one hand, the critical point theory of Szulkin \cite{szulkin1986minimax} for non-smooth functional (which provides a mountain pass theorem that applies to a large class of problems including those with supercritical nonlinearities) to derive the first component solution $u$. On the other hand, the second component solution $v$ is obtained by exploiting the so-called \textit{pointwise invariance property} established in \cite[Proposition 2.10]{boscaggin2023multiplicity}. Similar ideas were recently used in \cite{cowan2022supercritical,abbas2023existence} to derive the existence of solutions to systems of the form \eqref{a0} in bounded domains. See also \cite{alves2020super} where the same ideas were performed on supercritical Hamiltonian systems with Neumann boundary condition on unbounded domains of the form $\R^m\times B_r$ (where $B_r$ is a ball centered at the origin with radius $r$ in $\R^n$) to build solutions.  
%See \cite{moameni2017variational} for early use of this abstract result. 

Our next results deal with symmetry breaking and multiplicity of solutions of \eqref{e-1} obtained in Theorem \ref{first-main-result} when $a(x)=b(x)=1$, that is, a system of the form
\begin{equation}\label{e-2}
\left\{
\begin{aligned}
-\Delta u+u&=v^{p-1}~~~~\text{in}~~A_R\\
-\Delta v+v&=u^{q-1}~~~~\text{in}~~A_R\\
u, v&>0~~~~~~~~\text{in}~~A_R\\
u=v&=0~~~~~~~~\text{on}~~\partial A_R,
\end{aligned}
\right.
\end{equation}
Before stating our results, we first introduce the following optimal constant in Hardy inequality on $A_R$
\begin{equation}
    \Lambda_H:=\Lambda_H(A_R)=\inf_{0\neq\phi\in H^1_0(A_R)}\frac{\int_{A_R}|\nabla\phi|^2\ dx}{\int_{A_R}\frac{\phi^2}{|x|^2}\ dx}.
\end{equation}
The next result asserts that symmetry breaking occurs.
\begin{thm}\label{second-main-result}
    Let $3< N=m+n$ with $1< n\leq m$. Let $R>0$ and $q\geq p>2$. Suppose that $(u, v)$ is the solution of \eqref{e-2} obtained in Theorem \ref{first-main-result} which is invariant under the group action $O(m)\times O(n)$. If
    \begin{equation}\label{condition-for-symmetry-breaking}
        (p-1)(q-1)>\Big(1+\frac{2N}{\Lambda_H}\Big)^2\Big(\frac{q}{p}\Big),
    \end{equation}
    then $(u, v)$ is non-radial.
\end{thm}
Our result on multiplicity reads as follows.
\begin{thm}\label{third-main-result}
    For every $2\leq k\leq\Big\lfloor\frac{N}{2}\Big\rfloor$ and $q\geq p>2$, the problem \eqref{e-2} has $k-1$ positive distinct non-radial solutions if 
     \begin{equation}
        (p-1)(q-1)>\Big(1+\frac{2N}{\Lambda_H}\Big)^2\Big(\frac{q}{p}\Big),
    \end{equation}
    and either of the following two conditions holds:
    \begin{itemize}
        \item [$(i)$] $k>\frac{p+1}{p-1}$~and 
        \begin{equation*}
            \frac{1}{p}+\frac{1}{q}>1-\frac{2}{k+1}
        \end{equation*}
        or;
        \item [$(ii)$] $k\leq\frac{p+1}{p-1}$ and no lower bound condition imposed on $\frac{1}{p}+\frac{1}{q}$.
    \end{itemize}
\end{thm}
~\\
The paper is organized as follows. In Section \ref{section:preliminary} we recall some well-known properties of critical point theory due to Szulkin as well as Agmon-Douglas-Nirenberg's regularity result on smooth unbounded domains with bounded boundary. Section \ref{section:proof of first result} is devoted to the proof of Theorem \ref{first-main-result}. Finally, in Section \ref{section:non-radial solutions}, we are concerned with the proof of Theorems \ref{second-main-result} and \ref{third-main-result}. \\\\
\textbf{Notation:} Throughout the paper, $\lfloor x\rfloor$ denotes the floor of $x$. It is the greatest integer less than or equal to $x$.

\section{Preliminary}\label{section:preliminary}

In this section, we collect some preliminaries that will be useful throughout the paper. We start by recalling some well-known properties from critical point theory. Let $V$ be a real Banach space, $\Phi\in C^1(V,\R)$, and $\Psi: V\to(-\infty,+\infty]$ be a proper, convex and lower-semicontinuous function. Consider the functional $I: V\to(-\infty,+\infty]$ defined by
\begin{equation}\label{functional}
    I:=\Psi-\Phi.
\end{equation}
The following definition of critical point is due to Szulkin \cite{szulkin1986minimax}.
\begin{defi}\label{critical-point-szulking}
    A point $\overline{u}\in V$ is said to be a critical point of $I$ if $I(\overline{u})<\infty$ and if it satisfies the following inequality
    \begin{equation}
        \langle D\Phi(\overline{u}),\overline{u}-v\rangle+\Psi(v)-\Psi(\overline{u})\geq0,~~~\forall v\in V.
    \end{equation}
\end{defi}
We also recall the following notion of Palais-Smale compactness condition introduced in \cite{szulkin1986minimax}.
\begin{defi}\label{palais-smale-compactness-condition}
    We say that $I$ satisfies the Palais-Smale compactness condition (PS) if every sequence $\{u_i\}$ such that
    \begin{itemize}
        \item [$(i)$] $I(u_i)\to c\in \R$

        \item [$(ii)$] $\langle D\Phi(u_i), u_i-v\rangle+\Psi(v)-\Psi(u_i)\geq-\varepsilon_i\|u_i-v\|_V,~~~~\forall v\in V$
    \end{itemize}
    where $\varepsilon_i\to 0$, has a convergent subsequence.
\end{defi}
We shall also need a mountain pass theorem which is due to Szulkin \cite{szulkin1986minimax}. 
\begin{thm}[Mountain Pass Theorem]\label{mountain-pass-theorem}
    Let $I: V\to(-\infty,+\infty]$ be a functional of the form \eqref{functional} satisfying the Palais-Smale compactness condition and mountain pass geometry (MPG):
    \begin{itemize}
        \item [$(i)$] $I(0)=0$;

        \item [$(ii)$] There exists $e\in V$ such that $I(e)\leq0$;

        \item [$(iii)$] There exists $\rho$ such that $0<\rho<\|e\|$ and for every $u\in V$ with $\|u\|=\rho$ one has $I(u)>0$.
    \end{itemize}
    Then $I$ has a critical value $c$ which is characterized by
    \begin{equation}
        c:=\inf_{\gamma\in \Gamma}\sup_{t\in[0,1]}I(\gamma(t))
    \end{equation}
    where $\Gamma=\{\gamma\in C(V,\R): \gamma(0)=0,~\gamma(1)=e\}$
\end{thm}
The next result arising from convex analysis can be found in \cite{abbas2023existence}. We report its proof here for the sake of completeness.
\begin{lemma}\label{convex-analysis-result}
    Let $V$ be a reflexive Banach space and let $f: V\to\R$ be a convex and differentiable functional. If
    \begin{equation}\label{convex0}
        f(u)-f(\overline{u})\geq \langle Df(u), u-\overline{u}\rangle,
    \end{equation}
    then $Df(u)=Df(\overline{u})$, where $\langle\cdot, \cdot\rangle$ is the duality pairing between $V$ and $V^*$. In particular, if $f$ is strictly convex, then $u=\overline{u}$.
\end{lemma}
\begin{proof}
    Since $f$ is convex, then 
    \begin{equation}\label{convex1}
        f(\overline{u})-f(u)\geq\langle Df(u), \overline{u}-u\rangle.
    \end{equation}
     From \eqref{convex1} and \eqref{convex0}, it follows that
     \begin{equation}\label{in00}
         f(u)-f(\overline{u})=\langle Df(u), u-\overline{u}\rangle.
     \end{equation}
     For all $w\in V$, we also have
     \begin{equation*}
         f(w)-f(u)\geq\langle Df(u), w-u\rangle~~\text{that is,}~~f(w)\geq f(u)+\langle Df(u), w\rangle-\langle Df(u), u\rangle.
     \end{equation*}
     So,
     \begin{equation}\label{in0}
         f(w)-\langle Df(u), w\rangle\geq f(u)-\langle Df(u), u\rangle.
     \end{equation}
     We now define
     \begin{equation*}
         \phi(w)=f(w)-\langle Df(u), w\rangle.
     \end{equation*}
     Notice that $\phi$ is convex since does $f$. On the other hand, the inequality \eqref{in0} becomes
     \begin{equation}\label{in01}
         \phi(w)\geq\phi(u),~~~\forall w\in V.
     \end{equation}
     In particular, $u$ is a minimum of $\phi$ in $V$.
     Next, using \eqref{in00}, we have
     \begin{equation}\label{in02}
         \phi(\overline{u})=f(\overline{u})-\langle Df(u),\overline{u}\rangle=f(u)-\langle Df(u),u\rangle=\phi(u).
     \end{equation}
     From \eqref{in02} and \eqref{in01} we also deduce that $\phi$ attains its minimum at $w=\overline{u}$. In particular,
     \begin{equation*}
         0=D\phi(\overline{u})=Df(\overline{u})-Df(u),~~\text{that is,}~~Df(u)=Df(\overline{u}). 
     \end{equation*} 
     
     Now, if $f$ is strictly convex, then $\phi$ is also strictly convex. Consequently, $\phi$ has at most one minimum. Thus, $u$ must be equal to $\overline{u}$, that is, $u=\overline{u}$, as wanted. 
\end{proof}
We conclude this section with the following existence and regularity result on smooth unbounded domain with bounded boundary. This result is due to Agmon-Douglis-Nirenberg, see e.g., \cite[Theorem 9.32]{brezis2010functional}. 

\begin{lemma}\label{agmon-douglis-nirenberg-result}
    Let $t\in (1,\infty)$ and $f\in L^t(A_R)$. Then there exists a unique solution $u\in W^{2,t}(A_R)\cap W^{1,t}_0(A_R)$ of the problem
    \begin{equation}
        \left\{
        \begin{aligned}
            -\Delta u+u&=f~~~\text{in}~A_R\\
            u&=0~~~\text{on}~\partial A_R.
        \end{aligned}
        \right.
    \end{equation}
    Moreover,
    \begin{equation}
        \|u\|_{W^{2,t}(A_R)}\leq C\|f\|_{L^t(A_R)}
    \end{equation}
    where $C$ is a positive constant depending on $t$ and $A_R$.
\end{lemma}

\section{Proof of Theorem \ref{first-main-result}}\label{section:proof of first result}
The aim of this section is to prove the first main result of this paper, Theorem \ref{first-main-result}. For this, let us start by introducing an equivalent variational formulation of problem \eqref{e-1} via the so-called \textit{reduction-by-inversion} approach.

From the first equation on \eqref{e-1} we have
\begin{align}\label{v}
    v=|-\Delta u+u|^{\frac{1}{p-1}}a(x)^{-\frac{1}{p-1}}=|-\Delta u+u|^{p'-1}a(x)^{-(p'-1)}
\end{align}
where $\frac{1}{p}+\frac{1}{p'}=1$. Now, plugging \eqref{v} into the second equation of \eqref{e-1}, we obtain the following equivalent problem 
\begin{equation}\label{e-1-equiv}
    \left\{
    \begin{aligned}
        -\Delta\Big(|-\Delta u+u|^{p'-1}a(x)^{-(p'-1)}\Big)+|-\Delta u+u|^{p'-1}a(x)^{-(p'-1)}&=b(x)u^{q-1}~~~~\text{in}~A_R\\
        u&>0~~~~~~~~~~~~~~\text{in}~A_R\\
        u=\Delta u&=0~~~~~~~~~~~~~~\text{on}~\partial A_R. 
    \end{aligned}
    \right.
\end{equation}
The Euler-Lagrange functional associated to \eqref{e-1-equiv} is defined as
\begin{align}\label{euler-lagrange-function}
    I(u)&=\frac{1}{p'}\int_{A_R}\frac{|-\Delta u+u|^{p'}}{a(x)^{p'-1}}\ dx-\frac{1}{q}\int_{A_R}b(x)|u|^q\ dx\\
 \nonumber   &=\Psi(u)-\Phi(u)
\end{align}
where
\begin{equation*}
    \Psi(u)=\frac{1}{p'}\int_{A_R}\frac{|-\Delta u+u|^{p'}}{a(x)^{p'-1}}\ dx~~~\text{and}~~~\Phi(u)=\frac{1}{q}\int_{A_R}b(x)|u|^q\ dx. 
\end{equation*}
We consider the group action $G=O(m)\times O(n)$. Here, $O(k)$ represents the orthogonal group in $\R^k$. Define
\begin{equation*}
    W^{2,p'}_{G}(A_R)=\{u\in W^{2,p'}(A_R): gu=u,~\forall g\in G\}
\end{equation*}
where $gu(x)=u(g^{-1}x)$. Let $\cC$ be the cone defined by
\begin{align*}
    \cC&:=\{u\in W^{2,p'}_{G}(A_R): u=u(s,t), u\geq0, su_t-tu_s\leq0~\text{a.e. in}~\widehat{A}_R\}
\end{align*}
where $\widehat{A}_R:=\{(s,t)\in\R^2: s^2+t^2>R^2, s\geq0, t\geq0\}$. Notice that if we write $(s,t)\in\widehat{A}_R$ in polar coordinates as $s=r\cos\theta$ and $t=r\sin\theta$, then one can consider the cone $\cC$ as the set of functions $u=u(r,\theta) \geq0$ satisfying the inequality $u_{\theta}\leq0$ a.e. in $\widetilde{A}_R$, where $\widetilde{A}_R:=\{(r,\theta): r\in(R,\infty),~\theta\in[0,\frac{\pi}{2}]\}$. 

Let $K$ be the convex cone given by
\begin{equation*}
    K=\cC\cap H^1_0(A_R).
\end{equation*}
It is easily seen that $K$ is weakly closed with respect to the $W^{2,p'}(A_R)$-norm. We now introduce the restriction of $I$ on $K$ as
\begin{equation*}
    I_K=\Psi_K-\Phi
\end{equation*}
where
\begin{equation}\label{psi-k}
    \Psi_K(u):=\left\{
    \begin{aligned}
        &\Psi(u),~~~\text{if}~~u\in K\\
        &0~~~~~~~~~\text{otherwise}.
    \end{aligned}
    \right.
\end{equation}
Throughout the rest of the paper, we use $I$ in place of $I_K$ to alleviate the notations.

A crucial ingredient in our proof of Theorem \ref{first-main-result} is given by the following compact embedding. 
\begin{prop}\label{compact-embedding}
    Let $N=m+n$ with $1< n\leq m$. The embedding $K\hookrightarrow L^d(A_R)$ is compact for all $p'< d<\frac{(n+1)p'}{(n+1)-2p'}=\max\{(p')^*_{m+1}, (p')^*_{n+1}\}$ where $(p')^*_{m+1}=\frac{(m+1)p'}{(m+1)-2p'}$ and $(p')^*_{n+1}=\frac{(n+1)p'}{(n+1)-2p'}$.
\end{prop}
\begin{remark}\label{rmk2}
    Notice that in Proposition \ref{compact-embedding}, the range of $d$ can be understood as follows. 
    \begin{align*}
        &p'< d<\frac{(n+1)p'}{(n+1)-2p'}=\max\Big\{ \frac{(n+1)p'}{(n+1)-2p'}, \frac{(m+1)p'}{(m+1)-2p'} \Big\}~~~~\text{if}~~~(n+1)-2p'>0\\
        &p'< d<\infty~~~~~~~~~~~~~~~~~~~~~~~~~~~~~~~~~~~~~~~~~~~~~~~~~~~~~~~~~~~~~~~~~~~~~~\text{if}~~~(n+1)-2p'\leq0.
    \end{align*}
    Moreover, the above inequalities can be rewritten as 
    \begin{align}
        &\frac{1}{p}+\frac{1}{d}>1-\frac{2}{n+1}=\min\Big\{ 1-\frac{2}{n+1}, 1-\frac{2}{m+1}\Big\}~~~~\text{if}~~~~~ n>\frac{p+1}{p-1}\label{m0}\\ 
        &\text{no lower bound condition imposed on}~~\frac{1}{p}+\frac{1}{d}~~~~~~~~~~~\text{if}~~~~~ n\leq\frac{p+1}{p-1}\label{n0}.
    \end{align}
    Indeed, since $p'=\frac{p}{p-1}$, then
    \begin{equation*}
        (n+1)-2p'=(n+1)-\frac{2p}{p-1}=\frac{(n+1)(p-1)-2p}{p-1}=\frac{n(p-1)-p-1}{p-1}. 
    \end{equation*}
    So, 
    \begin{align}\label{m}
        (n+1)-2p'>0 \Longleftrightarrow n>\frac{p+1}{p-1}
    \end{align}
    and
    \begin{align}\label{n}
        (n+1)-2p'\leq0 \Longleftrightarrow n\leq\frac{p+1}{p-1}.
    \end{align}
    Now
    \begin{align}\label{o}
        \frac{(n+1)p'}{(n+1)-2p'}=\frac{(n+1)p}{(n+1)(p-1)-2p}=\frac{(n+1)p}{n(p-1)-p-1}.
    \end{align}
    Thus, from \eqref{m}, \eqref{n}, and \eqref{o}, we deduce \eqref{m0} and \eqref{n0}.
\end{remark}
\begin{remark}\label{rmk2_1}
    The restriction $1<n$ that is, $2\leq n$ that we require in Proposition \ref{compact-embedding} as well as in the whole paper is due to the fact that the compact embedding by Lion (see Lemma \ref{lions-compact-embedding}) holds for $N$ of the form $N=\sum_{i=1}^{l}N_i$ with $N_i \geq 2$. As a consequence of this, our main results are left to be open for $N=3$. 
\end{remark}
We now wish to prove Proposition \ref{compact-embedding}. To this end, we first collect some technical lemmas. 
\begin{lemma}\label{continuous-embedding}
    Let $N=m+n$ with $1< n\leq m$. If $(n+1)-2p'>0$, then, for all $p'\leq d\leq\frac{(n+1)p'}{(n+1)-2p'}=\max\{(p')^*_{m+1}, (p')^*_{n+1}\}$, there exists a positive constant $C=C(n,d)>0$ such that
    \begin{equation}\label{p0}
        \|u\|_{L^d(A_R)}\leq C\|u\|_{W^{2,p'}(A_R)}~~~\text{for all}~~u\in K.
    \end{equation}
    If $(n+1)-2p'\leq0$ then \eqref{p0} holds for all $d\in[p',\infty)$. 
\end{lemma}
\begin{proof}
    The proof of this lemma takes inspiration from \cite[Lemma 2.4]{boscaggin2023multiplicity}. 
    
    Let $u\in K$. We have
    \begin{equation*}
        \|u\|^d_{L^{d}(A_R)}=\int_{A_R}|u(x)|^d\ dx=\omega_{m-1}\omega_{n-1}\int_{\widehat{A}_R}|u(s,t)|^d s^{m-1} t^{n-1}\ dsdt,
    \end{equation*}
    where $\omega_{m-1}$ (resp. $\omega_{n-1}$) denotes the surface measure of the sphere $\mathbb{S}^{m-1}$ (resp. $\mathbb{S}^{n-1}$).
    
    Using polar coordinates $s=r\cos\theta,~t=r\sin\theta$, we have
    \begin{align}\label{ld}
     \nonumber   \|u\|^d_{L^{d}(A_R)}&=\omega_{m-1}\omega_{n-1}\int^{\infty}_{R}\int^{\frac{\pi}{2}}_{0}r^{m-1}(\cos\theta)^{m-1} r^{n-1} (\sin\theta)^{n-1}|u(r,\theta)|^d r\ d\theta dr\\
     \nonumber   &=\omega_{m-1}\omega_{n-1}\int^{\infty}_{R}\int^{\frac{\pi}{3}}_{0}r^{m-1}(\cos\theta)^{m-1} r^{n-1} (\sin\theta)^{n-1}|u(r,\theta)|^d r\ d\theta dr\\
     \nonumber   &+ \omega_{m-1}\omega_{n-1}\int^{\infty}_{R}\int^{\frac{\pi}{2}}_{\frac{\pi}{3}}r^{m-1}(\cos\theta)^{m-1} r^{n-1} (\sin\theta)^{n-1}|u(r,\theta)|^d r\ d\theta dr\\
        &=: J^1_R+J^2_R
    \end{align}
 where
 \begin{align*}
     J^1_R=\omega_{m-1}\omega_{n-1}\int^{\infty}_{R}\int^{\frac{\pi}{3}}_{0}r^{m-1}(\cos\theta)^{m-1} r^{n-1} (\sin\theta)^{n-1}|u(r,\theta)|^d r\ d\theta dr
 \end{align*}
and
\begin{align*}
    J^2_R=\omega_{m-1}\omega_{n-1}\int^{\infty}_{R}\int^{\frac{\pi}{2}}_{\frac{\pi}{3}}r^{m-1}(\cos\theta)^{m-1} r^{n-1} (\sin\theta)^{n-1}|u(r,\theta)|^d r\ d\theta dr.
\end{align*}
To estimate $J^2_R$, we first observe that for $\theta\in [\frac{\pi}{3},\frac{\pi}{2}]$, there holds that $\sin\theta\leq c\sin(\theta-\frac{\pi}{4})$ for some $c>0$. Using this and the fact that $\theta\mapsto u(r,\theta)$ and $\theta\mapsto\cos\theta$ are decreasing, we have 
\begin{align*}
    J^2_R&\leq c^{n-1}\omega_{m-1}\omega_{n-1}\int^{\infty}_{R}\int^{\frac{\pi}{2}}_{\frac{\pi}{3}}r^{m-1}\Big(\cos(\theta-\frac{\pi}{4})\Big)^{m-1} r^{n-1} \Big(\sin(\theta-\frac{\pi}{4})\Big)^{n-1}u(r,\theta-\frac{\pi}{4})^d r \\
    &=c^{n-1}\omega_{m-1}\omega_{n-1}\int^{\infty}_{R}\int^{\frac{\pi}{4}}_{\frac{\pi}{12}}r^{m-1}(\cos\theta)^{m-1} r^{n-1} (\sin\theta)^{n-1}u(r,\theta)^d r\ d\theta dr\\
    &\leq c^{n-1}\omega_{m-1}\omega_{n-1}\int^{\infty}_{R}\int^{\frac{\pi}{3}}_{0}r^{m-1}(\cos\theta)^{m-1} r^{n-1} (\sin\theta)^{n-1}u(r,\theta)^d r\ d\theta dr\\
    &=c^{n-1}J^1_R.
\end{align*}
So
\begin{equation}\label{j}
    J^2_R\leq c^{n-1} J^1_R.
\end{equation}
We now turn our attention to $J^1_R$. We have
\begin{align*}
     J^1_R&=\omega_{m-1}\omega_{n-1}\int^{\infty}_{R}\int^{\frac{\pi}{3}}_{0}r^{m-1}(\cos\theta)^{m-1} r^{n-1} (\sin\theta)^{n-1}|u(r,\theta)|^d r\ d\theta dr\\
     &=\omega_{m-1}\omega_{n-1}\int_{\Big\{\substack{(s,t)\in\R^2:~ s^2+t^2>R^2,~ s\geq0, ~t\geq0,\\ \arctan(\frac{t}{s})\in (0, \frac{\pi}{3})} \Big\}}|u(s,t)|^d s^{m-1} t^{n-1}\ dsdt.
 \end{align*}
We write $u(s,t)=u(s,y)$ with $|y|=t$. then
\begin{align*}
     J^1_R&=\omega_{m-1}\int_{\Big\{\substack{(s,y)\in\R^1\times\R^{N-m}:~ s^2+|y|^2>R^2,~ s\geq0,\\ \arctan(\frac{|y|}{s})\in (0, \frac{\pi}{3})} \Big\}}|u(s,y)|^d s^{m-1}\ dsdy\\
     &=\omega_{m-1}\int_{\Pi_R}|u(s,y)|^d s^{m-1}\ ds dy
 \end{align*}
where
\begin{equation*}
    \Pi_R:=\left\{
        (s,y)\in\R^1\times\R^{N-m}:~ s^2+|y|^2>R^2,~ s\geq0,~~\arctan\Big(\frac{|y|}{s}\Big)\in \Big(0,\frac{\pi}{3}\Big) 
    \right\}.
\end{equation*}
It is easily seen that
\begin{equation}\label{lower-bound-for-s}
    s>\frac{R}{2},~~~\forall (s,y)\in \Pi_R.
\end{equation}
We now set 
\begin{equation*}
    v(s,y)=s^{\frac{m-1}{d}}u(s,y).
\end{equation*}
From the Sobolev embedding $W^{2,p'}(\Pi_R)\hookrightarrow L^d(\Pi_R)$ in $\R^{N-m+1}=\R^{n+1}$, we have 
\begin{align*}
    &J^1_R=\omega_{m-1}\int_{\Pi_R}|v(s,y)|^d \ ds dy\\
    &\leq C_1\Bigg(\int_{\Pi_R}|D^2v(s,y)|^{p'}+|\nabla v(s,y)|^{p'}+|v(s,y)|^{p'}\ dsdy\Bigg)^{\frac{d}{p'}}\\
    &\leq C_2 \Bigg(\int_{\Pi_R}\Big[s^{\frac{(m-1)p'}{d}}|D^2u(s,y)|^{p'}+\Big(s^{\frac{(m-1)p'}{d}-p'}+s^{\frac{(m-1)p'}{d}}\Big)|\nabla u(s,y)|^{p'}\\
    &~~~~~~~~~~~~~~~+\Big(s^{\frac{(m-1)p'}{d}-2p'}+s^{\frac{(m-1)p'}{d}-p'}\Big)|u(s,y)|^{p'}\Big]\ dsdy\Bigg)^{\frac{d}{p'}}\\
    &=C_2 \Bigg(\int_{\Pi_R}\Big[s^{\frac{(m-1)p'}{d}-m+1}|D^2u(s,y)|^{p'}+\Big(s^{\frac{(m-1)p'}{d}-p'-m+1}+s^{\frac{(m-1)p'}{d}-m+1}\Big)|\nabla u(s,y)|^{p'}\\
    &~~~~~~~~~~~~~~~+\Big(s^{\frac{(m-1)p'}{d}-2p'-m+1}+s^{\frac{(m-1)p'}{d}-p'-m+1}\Big)|u(s,y)|^{p'}\Big]s^{m-1}\ dsdy\Bigg)^{\frac{d}{p'}}\\
    &=C_2 \Bigg(\int_{\Pi_R}\Big[s^{\beta_1}|D^2u(s,y)|^{p'}+\Big(s^{\beta_2}+s^{\beta_1}\Big)|\nabla u(s,y)|^{p'}\\
    &\quad\quad\quad\quad\quad\quad\quad\quad\quad\quad\quad\quad\quad+\Big(s^{\beta_3}+s^{\beta_2}\Big)|u(s,y)|^{p'}\Big]s^{m-1}\ dsdy\Bigg)^{\frac{d}{p'}}
\end{align*}
where
\begin{align*}
    \beta_1&:=\frac{(m-1)p'}{d}-m+1\\
    \beta_2&:=\frac{(m-1)p'}{d}-p'-m+1\\
    \beta_3&:=\frac{(m-1)p'}{d}-2p'-m+1.
\end{align*}
Recalling \eqref{lower-bound-for-s} and using that $\beta_3<\beta_2<\beta_1\leq0$, there holds 
\begin{align}\label{x}
 \nonumber  J^1_R&\leq C_3\Bigg(\int_{\Pi_R}\Big(|D^2u|^{p'}+|\nabla u|^{p'}+|u|^{p'}\Big)s^{m-1}\ ds dy \Bigg)^{\frac{d}{p'}}\\
\nonumber &\leq C_4\Bigg(\int_{A_R}\Big(|D^2u|^{p'}+|\nabla u|^{p'}+|u|^{p'}\Big)\ dz dy \Bigg)^{\frac{d}{p'}}~~~\text{where}~~|z|=s\\
 &=C_4\|u\|^{d}_{W^{2,p'}(A_R)}.
\end{align}
Here, $C_i~(i\in\{1,2,3,4\})$ are generic positive constants. From \eqref{x}, \eqref{j}, and \eqref{ld} we conclude the proof.
\end{proof}

\begin{lemma}[\cite{lions1982symetrie}]\label{lions-compact-embedding}
    $W^{2,p'}_G(A_R)$ is compactly embedded in $L^d(A_R)$ for all $d\in (p', (p')^*_N)$ where $(p')^*_N:=\frac{Np'}{N-2p'}$. 
\end{lemma}

We are now ready to prove Proposition \ref{compact-embedding}.

\begin{proof}[Proof of Proposition \ref{compact-embedding}]
    Let $(u_i)\subset K$ be a bounded sequence with respect to the $W^{2,p'}(A_R)$-norm. Then, from Lemmas \ref{continuous-embedding} and \ref{lions-compact-embedding}, there exists $u\in K$ such that
    \begin{equation}\label{k}
        \begin{aligned}
            &u_i\rightharpoonup u~~~\text{weakly in}~~W^{2,p'}(A_R)\\
            &u\to u~~~\text{strongly in}~~~L^{\tau}(A_R)
        \end{aligned}
    \end{equation}
    for all $\tau\in (p', (p')^*_N)$. We consider the case $(n+1)-2p'>0$. We recall that $(p')^*_{n+1}>(p')^*_N$. Now, for fix $\tau\in (p', (p')^*_N)$ and $d\in [(p')^*_N, (p')^*_{n+1})$, interpolation inequality yields 
    \begin{equation}\label{l}
        \|u_i-u\|_{L^d(A_R)}\leq \|u_i-u\|^{\alpha}_{L^{\tau}(A_R)}\|u_i-u\|^{1-\alpha}_{L^{(p')^*_{n+1}}(A_R)}
    \end{equation}
    for some $\alpha\in (0,1)$ with $\frac{1}{d}=\frac{\alpha}{\tau}+\frac{1-\alpha}{(p')^*_{n+1}}$. Since by Lemma \ref{continuous-embedding} the quantity $\|u_i-u\|^{1-\alpha}_{L^{(p')^*_{n+1}}(A_R)}$ is uniformly bounded, then letting $i\to\infty$ in \eqref{l} and using \eqref{k}, we find that
    \begin{equation*}
        u_i\to u~~~\text{strongly in}~~L^d(A_R).
    \end{equation*}
    In the case when $(n+1)-2p'\leq0$, we fix $\tau\in (p', (p')^*_N)$ and $d\in[(p')^*_N,\infty)$. Then, as above, we apply again interpolation inequality to get
    \begin{equation}\label{p1}
        \|u_i-u\|_{L^d(A_R)}\leq\|u_i-u\|^{\alpha}_{L^{\tau}(A_R)}\|u_i-u\|^{1-\alpha}_{L^{2d}(A_R)}
    \end{equation}
    with $\frac{1}{d}=\frac{\alpha}{\tau}+\frac{1-\alpha}{2d}$. Using that $\|u_i-u\|^{1-\alpha}_{L^{2d}(A_R)}$ is bounded, thanks to Lemma \ref{continuous-embedding}, the proof is completed by sending $i\to\infty$ in \eqref{p1} and by using \eqref{k}.
\end{proof}

As a direct consequence of Proposition \ref{compact-embedding}, the $C^2$-functional $I$ introduced in \eqref{euler-lagrange-function} is well-defined on the Banach space 
\begin{equation*}
    V=W^{2,p'}(A_R)\cap H^1_0(A_R)\cap L^q(A_R).
\end{equation*}
Moreover, the following two-sided inequality also holds
\begin{equation}\label{two-sided-inequality}
    \|u\|_{W^{2,p'}(A_R)}\leq \|u\|_V\leq C\|u\|_{W^{2,p'}(A_R)},~~~~\forall u\in K.
\end{equation}
Exploiting the above two-sided inequality, we prove the following result which asserts that the functional $I$ has a critical point. 
\begin{lemma}\label{critical-point}
    The functional $I$ has a critical point $\overline{u}\in K$ in the sense of Definition \ref{critical-point-szulking}.
\end{lemma}
  
\begin{proof}
    We first notice that by Lemma \ref{agmon-douglis-nirenberg-result}, and using that $a$ is bounded from above and below, 
\begin{equation}
    \|u\|_{W^{2,p'}(A_R)}:=\Bigg( \int_{A_R}\frac{|-\Delta u+u|^{p'}}{a(x)^{p'-1}}\ dx \Bigg)^{\frac{1}{p'}}
\end{equation}
is an equivalent norm on $ W^{2,p'}(A_R)\cap H^1_0(A_R)$. The rest of the proof is divided into two steps:\\\\
\textbf{Step 1:} Verifying the Palais-Smale compactness condition.

Let $\{u_i\}\subset K$ be a sequence such that $I(u_i)\to c$ and
\begin{equation}\label{a}
    \langle D\Phi(u_i), u_i-v\rangle+\Psi(v)-\Psi(u_i)\geq -\varepsilon_i\|u_i-v\|_V~~\forall v\in V
\end{equation}
where $\varepsilon_i\to0$. Then, for $i$ sufficiently large,
\begin{align}\label{b}
  \nonumber  I(u_i)&=\frac{1}{p'}\int_{A_R}\frac{|-\Delta u_i+u_i|^{p'}}{a(x)^{p'-1}}\ dx-\frac{1}{q}\int_{A_R}b(x)|u_i|^q\ dx\\
    &=\frac{1}{p'}\|u_i\|^{p'}_{W^{2,p'}(A_R)}-\frac{1}{q}\int_{A_R}b(x)|u_i|^q\ dx\leq c+1.
\end{align}
Set $v_i=ru_i$ with $r=1+\frac{\delta}{q}$ where $\delta>0$ will be choose later. Substituting $v=v_i$ in \eqref{a}, we have
\begin{align*}
    \frac{r^{p'}-1}{p'}\int_{A_R}\frac{|-\Delta u_i+u_i|^{p'}}{a(x)^{p'-1}}\ dx+(1-r)\int_{A_R}b(x)|u_i|^{q}\geq-\varepsilon_i(r-1)\|u_i\|_{V}
\end{align*}
that is
\begin{align}\label{c}
    \frac{1-r^{p'}}{p'}\|u_i\|^{p'}_{W^{2,p'}(A_R)}+(r-1)\int_{A_R}b(x)|u_i|^{q}\leq\varepsilon_i(r-1)\|u_i\|_{V}.
\end{align}
Multiplying \eqref{b} by $\delta$ and summing up with \eqref{c} give
\begin{align*}
    \Big(\frac{\delta+1-r^{p'}}{p'}\Big)\|u_i\|^{p'}_{W^{2,p'}(A_R)}&\leq\delta c_0+\delta+\varepsilon_i\frac{\delta}{q}\|u_i\|_V\\
    &\leq \delta c_0+\delta+\varepsilon_iC\frac{\delta}{q}\|u_i\|_{W^{2,p'}(A_R)}.
\end{align*}
In the latter, we have used \eqref{two-sided-inequality}. Now, by choosing $\delta$ such that%\footnote{Such $\delta$ exists since $q>2>p'$.}
\begin{equation*}
    \delta+1>\Big(1+\frac{\delta}{q}\Big)^{p'}
\end{equation*}
there holds that
\begin{equation*}
    \|u_i\|_{W^{2,p'}(A_R)}\leq C_1(1+C_2\|u_i\|_{W^{2,p'}(A_R)}).
\end{equation*}
Thus, $\{u_i\}$ is a bounded sequence in $W^{2,p'}(A_R)$. So, after passing to a subsequence, there exists $u\in W^{2,p'}(A_R)$ such that 
\begin{equation*}
    \begin{aligned}
        &u_i\rightharpoonup u~~\text{weakly in}~~W^{2,p'}(A_R)\\
        &u_i\to u~~\text{strongly in}~~L^{q}(A_R)\\
        &u_i\to u~~\text{a.e. in}~~A_R.
    \end{aligned}
\end{equation*}
Notice that the second limit follows from the compact embedding $K\hookrightarrow L^{q}(A_R)$ (thanks to Proposition \ref{compact-embedding}). Next, by substitution $v=u$ in \eqref{b}, we get
\begin{align*}
    \frac{1}{p'}\Big(\|u\|^{p'}_{W^{2,p'}(A_R)}-\|u_i\|^{p'}_{W^{2,p'}(A_R)}\Big)+\int_{A_R}b(x)(u_i-u)\ dx\geq-\varepsilon_i\|u_i-u\|_V
\end{align*}
from which we deduce that
\begin{equation*}
    \frac{1}{p'}\Big(\limsup_{i\to\infty}\|u_i\|^{p'}_{W^{2,p'}(A_R)}-\|u\|^{p'}_{W^{2,p'}(A_R)}\Big)\leq0
\end{equation*}
and thus
\begin{equation*}
    u_i\to u~~\text{strongly in}~~W^{2,p'}(A_R).
\end{equation*}
This implies that $u_i\to u~~\text{strongly in}~~V$, as wanted.\\\\
\textbf{Step 2:} Verifying the mountain pass geometry of Theorem \ref{mountain-pass-theorem}.

By definition, we have $I(0)=0$ and thus condition $(i)$ in Theorem \ref{mountain-pass-theorem} is satisfied. Regarding condition $(ii)$, we proceed as follows. Let $u\in K$. Then for all $t\geq0$,
\begin{align*}
    I(tu)&=\frac{t^{p'}}{p'}\int_{A_R}\frac{|-\Delta u+u|^{p'}}{a(x)^{p'-1}}\ dx-\frac{t^q}{q}\int_{A_R}b(x)|u|^{q}\ dx\\
    &=\frac{t^{p'}}{p'}\|u\|^{p'}_{W^{2,p'}(A_R)}-\frac{t^q}{q}\int_{A_R}b(x)|u|^q\ dx.
\end{align*}
Since $q>2>p'$, then for $t$ sufficiently large, we have $I(tu)<0$. So, condition $(ii)$ follows for $e=tu$. We now turn our attention to condition $(iii)$ in Theorem \ref{mountain-pass-theorem}. Consider $u\in V$ such that $\|u\|_V=\rho>0$. If $u\notin K$, then $I(u)>0$ thanks to the definition of $\psi_K$ (see \eqref{psi-k}). Now, if $u\in K$,
\begin{align*}
    I(u)&=\frac{1}{p'}\|u\|^{p'}_{W^{2,p'}(A_R)}-\frac{1}{q}\int_{A_R}b(x)|u|^q\ dx\\
    &\geq\frac{c_1}{p'}\|u\|^{p'}_V-\frac{c_2}{q}\|u\|^q_V.
\end{align*}
In the latter, we have used \eqref{two-sided-inequality} and the fact that 
\begin{equation*}
    \int_{A_R}b(x)|u|^q\ dx\leq c_2\|u\|^q_{V}.
\end{equation*}
Thus
\begin{equation*}
    I(u)\geq \frac{c_1}{p'}\rho^{p'}-\frac{c_2}{q}\rho^q.
\end{equation*}
Using again that $q>2>p'$ it follows that for $\rho$ sufficiently small, $I(u)>0$. This proves condition $(iii)$.

Consequently, by the mountain pass theorem (see Theorem \ref{mountain-pass-theorem}), $I$ has a critical point $\overline{u}$ characterized by
\begin{equation*}
    I(\overline{u})=\inf_{\gamma\in \Gamma}\sup_{t\in [0,1]}I(\gamma(t))=:c
\end{equation*}
where $\Gamma=\{\gamma\in C([0,1],V): \gamma(0)=0,~\gamma(1)=e,~I(\gamma(1))\leq0\}$.
\end{proof}
The next proposition follows from \cite[Proposition 2.10]{boscaggin2023multiplicity}.
\begin{prop}\label{point-wise-invariance-single-equation}
Let $w(x)$ satisfy assumption $(\cH)$ and let $0\leq\widehat{u}\in H^1_{0,G}(A_R)\cap L^{d}(A_R)$ with $s\widehat{u}_t-t\widehat{u}_s\leq0$ a.e. in $\widehat{A}_R$ where $d>2$. Suppose that $\widehat{v}$ is the solution of 
    \begin{equation}
        \left\{
        \begin{aligned}
            -\Delta \widehat{v}+\widehat{v}&=w(x)\widehat{u}^{d-1}~~\text{in}~~A_R\\
            \widehat{v}&=0~~~~~~~~~~~~\text{on}~~\partial A_R.
        \end{aligned}
        \right.
    \end{equation}
    Then, $0\leq\widehat{v}\in H^1_{0,G}(A_R)\cap L^{d}(A_R)$ with $s\widehat{v}_t-t\widehat{v}_s\leq0$ a.e. in $\widehat{A}_R$.
\end{prop}
We shall also need the following proposition which links the critical points of $I$ to the solutions of the system \eqref{e-1}. Its proof is inspired by an argument in \cite{moameni2020critical}. 
\begin{prop}\label{existence-of-coulpe-of-solution}
    Let $\overline{u}$ be the critical point of $I$ obtained in Lemma \ref{critical-point}. Assume that there exists $\widehat{u}\in K$ and $\widehat{v}\in W^{2,q'}(A_R)\cap H^1_0(A_R)$ such that
    \begin{equation}\label{d}
        \left\{
        \begin{aligned}
            -\Delta \widehat{u}+\widehat{u}&=a(x)\widehat{v}^{p-1}~~~\text{in}~~A_R\\
            -\Delta \widehat{v}+\widehat{v}&=b(x)\overline{u}^{q-1}~~~\text{in}~~A_R.
        \end{aligned}
        \right.
    \end{equation}
    Then $(\widehat{u},\widehat{v})$ is a nontrivial and positive solution of 
    \begin{equation}\label{e}
        \left\{
        \begin{aligned}
            -\Delta u+u&=a(x)v^{p-1}~~~\text{in}~~A_R\\
            -\Delta v+v&=b(x)u^{q-1}~~~\text{in}~~A_R.
        \end{aligned}
        \right.
    \end{equation}
\end{prop}

\begin{proof}
    We introduce the functional $J: W^{2,p'}(A_R)\cap H^1_0(A_R)\to\R$ defined by
    \begin{equation*}
        J(w)=\frac{1}{p'}\int_{A_R}\frac{|-\Delta w+w|^{p'}}{a(x)^{p'-1}}\ dx-\int_{A_R}b(x)|\overline{u}|^{q-1}w\ dx.
    \end{equation*}
    We claim that $\widehat{u}$ is a critical point of $J$. Indeed, from \eqref{d} we have
    \begin{equation}\label{f}
        \left\{
        \begin{aligned}
            \widehat{v}&=\frac{1}{a(x)^{p'-1}}|-\Delta \widehat{u}+\widehat{u}|^{p'-1},~~~~~~p'=\frac{p}{p-1}\\
            \overline{u}&=\frac{1}{b(x)^{q'-1}}|-\Delta \widehat{v}+\widehat{v}|^{q'-1},~~~~~~~q'=\frac{q}{q-1}.
        \end{aligned}
        \right.
    \end{equation}
    Now, for all $\phi\in W^{2,p'}(A_R)\cap H^1_0(A_R)$,
    \begin{align*}
        \langle J'(\widehat{u}),\phi&\rangle~=\int_{A_R}\frac{1}{a(x)^{p'-1}}|-\Delta \widehat{u}+\widehat{u}|^{p'-1}(-\Delta\phi+\phi)\ dx-\int_{A_R}b(x)\overline{u}^{q-1}\phi\ dx\\
        &\stackrel{\eqref{f}}{=} \int_{A_R}\widehat{v}(-\Delta\phi+\phi)\ dx-\int_{A_R}b(x)|\overline{u}|^{q-1}\phi\ dx\\
        &~=\int_{A_R}(-\Delta\widehat{v}+\widehat{v})\phi\ dx-\int_{A_R}b(x)|\overline{u}|^{q-1}\phi\ dx\\
        &\stackrel{\eqref{d}}{=}\int_{A_R}b(x)\overline{u}^{q-1}\phi\ dx-\int_{A_R}b(x)\overline{u}^{q-1}\phi\ dx\\
        &~~=0.
    \end{align*}
    Thus $\widehat{u}$ is a critical point of $J$. In particular, for $\phi=\widehat{u}-\overline{u}$, there holds
    \begin{align*}
        0=\langle J'(\widehat{u}), \widehat{u}-\overline{u}\rangle&=\int_{A_R}\frac{1}{a(x)^{p'-1}}|-\Delta\widehat{u}+\widehat{u}|^{p'-1}(-\Delta(\widehat{u}-\overline{u})+\widehat{u}-\overline{u})\ dx\\
        &~~~~-\int_{A_R}b(x)\overline{u}^{q-1}(\widehat{u}-\overline{u})\ dx
    \end{align*}
    that is
    \begin{equation}\label{g}
        \int_{A_R}\frac{1}{a(x)^{p'-1}}|-\Delta\widehat{u}+\widehat{u}|^{p'-1}(-\Delta(\widehat{u}-\overline{u})+\widehat{u}-\overline{u})\ dx=\int_{A_R}b(x)\overline{u}^{q-1}(\widehat{u}-\overline{u})\ dx.
    \end{equation}
    Now using that $\overline{u}$ is a critical point of $I$, we have by Definition \ref{critical-point-szulking} that 
    \begin{equation}\label{h}
        \frac{1}{p'}\int_{A_R}\frac{|-\Delta w+w|^{p'}}{a(x)^{p'-1}}\ dx-\frac{1}{p'}\int_{A_R}\frac{|-\Delta\overline{u}+\overline{u}|^{p'}}{a(x)^{p'-1}}\ dx\geq \int_{A_R}b(x)\overline{u}^{q-1}(w-\overline{u})\ dx,~~\forall w\in K.
    \end{equation}
    Letting in particular $w=\widehat{u}$ in \eqref{h}, we obtain, thanks to \eqref{g}, that
    \begin{align*}
        &\frac{1}{p'}\int_{A_R}\frac{|-\Delta \widehat{u}+\widehat{u}|^{p'}}{a(x)^{p'-1}}\ dx-\frac{1}{p'}\int_{A_R}\frac{|-\Delta\overline{u}+\overline{u}|^{p'}}{a(x)^{p'-1}}\ dx\\
        &~~~~~\geq \int_{A_R}\frac{1}{a(x)^{p'-1}}|-\Delta\widehat{u}+\widehat{u}|^{p'-1}(-\Delta(\widehat{u}-\overline{u})+\widehat{u}-\overline{u})\ dx.
    \end{align*}
    Now, by Lemma \ref{convex-analysis-result}, it follows that $\widehat{u}=\overline{u}$. The proof is therefore completed considering $\overline{u}=\widehat{u}$ in \eqref{d}. Notice that the nontriviality of $(\widehat{u},\widehat{v})$ follows from the fact that the critical point $\overline{u}$ is nontrivial. On the other hand, the positivity of $(\widehat{u},\widehat{v})$ is a consequence of the strong maximum principle.
\end{proof}
We are now in position to prove Theorem \ref{first-main-result}.

\begin{proof}[Proof of Theorem \ref{first-main-result}~(completed)]
    We first recall that by Proposition \ref{compact-embedding}, the embedding $K\hookrightarrow L^q(A_R)$ is compact for
    \begin{align*}
        &p'< q<\frac{(n+1)p'}{(n+1)-2p'}~~~\text{if}~~(n+1)-2p'>0\\
        &p'< q<\infty\quad\quad\quad\quad\quad\quad\text{if}~~(n+1)-2p'\leq0. 
    \end{align*}
    This can be rewritten, thanks to Remark \ref{rmk2}, as
    \begin{align*}
        &\frac{1}{p}+\frac{1}{q}>1-\frac{2}{n+1}=\min\Big\{1-\frac{2}{n+1},1-\frac{2}{m+1}\Big\}~~~\text{if}~~n>\frac{p+1}{p-1}\\
        &\text{no lower bound condition on}~~\frac{1}{p}+\frac{1}{q}\quad\quad\quad\quad\quad\quad\quad~\text{if}~~n\leq\frac{p+1}{p-1}.
    \end{align*}
    Now, it follows from Lemma \ref{critical-point} that $I$ has a critical point $\overline{u}$ characterized by
    \begin{equation*}
    I(\overline{u})=\inf_{\gamma\in \Gamma}\sup_{t\in [0,1]}I(\gamma(t))=:c
\end{equation*}
where $\Gamma=\{\gamma\in C([0,1],V): \gamma(0)=0\neq\gamma(1),~I(\gamma(1))\leq0\}$. By Proposition \ref{point-wise-invariance-single-equation} there exists $\widehat{v}\in K$ satisfying
\begin{equation}
    -\Delta\widehat{v}+\widehat{v}=\overline{b}(x)\overline{u}^{q-1}~~\text{in}~~A_R
\end{equation}
where $\overline{b}$ satisfies assumption $(\cH)$. On the other hand, applying again Proposition \ref{point-wise-invariance-single-equation}, we find $\widehat{u}\in K$ satisfying
\begin{equation}
    -\Delta\widehat{u}+\widehat{u}=\overline{a}(x)\widehat{v}^{p-1}~~\text{in}~~A_R
\end{equation}
where $\overline{a}$ satisfies assumption $(\cH)$. In summary, there exists $(\widehat{u},\widehat{v})$ satisfying 
\begin{equation*}
    \left\{
    \begin{aligned}
        -\Delta\widehat{u}+\widehat{u}&=a(x)\widehat{v}^{p-1}~~\text{in}~~A_R\\
        -\Delta\widehat{v}+\widehat{v}&=b(x)\overline{u}^{q-1}~~\text{in}~~A_R.
    \end{aligned}
    \right.
\end{equation*}
It then follows from Proposition \ref{existence-of-coulpe-of-solution} that $(\widehat{u},\widehat{v})$ is a nontrivial and positive solution of 
\begin{equation*}
    \left\{
    \begin{aligned}
        -\Delta u+u&=a(x) v^{p-1}~~\text{in}~~A_R\\
        -\Delta v+v&=b(x) u^{q-1}~~\text{in}~~A_R.  
    \end{aligned}
    \right.
\end{equation*}
This finishes the proof.
\end{proof}

\section{Existence of non-radial solutions}\label{section:non-radial solutions}

In this section, we analyze the symmetry breaking and multiplicity result of solutions of \eqref{e-1} obtained in Theorem \ref{first-main-result} when $a(x)=b(x)=1$, that is, 
\begin{equation}\label{e-1-radial}
\left\{
\begin{aligned}
-\Delta u+u&=v^{p-1}~~~~\text{in}~~A_R\\
-\Delta v+v&=u^{q-1}~~~~\text{in}~~A_R\\
u, v&>0~~~~~~~~~\text{in}~~A_R\\
u=v&=0~~~~~~~~~\text{on}~~\partial A_R.
\end{aligned}
\right.
\end{equation}
To this end, the following theorem is of key importance.
\begin{thm}\label{t}
    Let $q\geq p\geq2$ and let $(u,v)$ be a solution of \eqref{e-1-radial}. Then
    \begin{equation}\label{t0}
        \inf_{0\neq\phi\in H^1_0(A_R)}\frac{\int_{A_R}|\nabla\phi|^2\ dx+\int_{A_R}|\phi|^2\ dx}{\int_{A_R}\phi^2v(x)^{\frac{p-2}{2}}u(x)^{\frac{q-2}{2}}\ dx}\leq \sqrt{\frac{q}{p}}.
    \end{equation}
\end{thm}
\begin{proof}
To prove that the identity \eqref{t0} holds true, we start by establishing the following pointwise comparison property between components of solutions
\begin{equation}\label{claim}
    qv(x)^p\geq pu(x)^q,~~\forall x\in A_R.
\end{equation}
We follow the ideas of \cite[Lemma 2.7]{souplet2009proof}. Set $l=\frac{p}{q}\in(0,1]$ and define $w(x)=u(x)-\tau v(x)^l$ with $\tau=l^{-\frac{1}{q}}$. Then a simple calculation shows that 
    \begin{align*}
        \Delta w&=\Delta u-\tau lv^{l-1}\Delta v-\tau l(l-1)v^{l-2}|\nabla v|^2\\ 
        &\geq \Delta u-\tau lv^{l-1}\Delta v\\
        &=u-v^{p-1}+\tau lv^{l-1}(u^{q-1}-v)\\
        &=\tau lv^{l-1}u^{q-1}-v^{p-1}+u-\tau lv^{l}\\
        &\geq \tau lv^{l-1}u^{q-1}-v^{p-1}+u-\tau v^{l}~~~~\text{as}~~l\leq1,
    \end{align*}
    that is,
    \begin{align*}
        \Delta w-w&\geq \tau lv^{l-1}u^{q-1}-v^{p-1}\\
        &=v^{l-1}\Big(\frac{u^{q-1}}{\tau^{q-1}}-v^{l(q-1)}\Big).
    \end{align*}
    From the identity above, there holds
    \begin{equation*}
        \Delta w-w\geq0~~\text{in the set}~~\{x\in A_R: w(x)\geq0\}.
    \end{equation*}
    In particular, for every $\varepsilon>0$ sufficiently small, we have
    \begin{equation*}
        (w-\varepsilon)^+(\Delta w-w)\geq0.
    \end{equation*}
    Integrating the above identity over $A_R$ we get 
    \begin{equation*}
        \int_{A_R}|\nabla(w-\varepsilon)^+|^2\ dx+\int_{A_R}|(w-\varepsilon)^+|^2\ dx+\varepsilon\int_{A_R}(w-\varepsilon)^+\ dx\leq0.
    \end{equation*}
    So, $(w-\varepsilon)^+=0$, that is, $w\leq\varepsilon$. Since $\varepsilon$ was chosen arbitrarily small, the inequality \eqref{claim} follows. 

    Now, from \eqref{claim}, we deduce that
    \begin{equation}
        v^2v^{\frac{p-2}{2}}u^{\frac{q-2}{2}}\geq\sqrt{\frac{p}{q}}u^{q-1}v.
    \end{equation}
    Hence,
    \begin{align*}
        \inf_{0\neq\phi\in H^1_0(A_R)}\frac{\int_{A_R}|\nabla\phi|^2\ dx+\int_{A_R}|\phi|^2\ dx}{\int_{A_R}\phi^2v(x)^{\frac{p-2}{2}}u(x)^{\frac{q-2}{2}}\ dx}&\leq \frac{\int_{A_R}|\nabla v|^2\ dx+\int_{A_R}|v|^2\ dx}{\int_{A_R}v^2v(x)^{\frac{p-2}{2}}u(x)^{\frac{q-2}{2}}\ dx}\\
        &\leq\frac{\int_{A_R}u^{q-1}v\ dx}{\sqrt{\frac{p}{q}}\int_{A_R}u^{q-1}v\ dx}=\sqrt{\frac{q}{p}}.
    \end{align*}
\end{proof}
We are now ready to prove Theorem \ref{second-main-result}.
\begin{proof}[Proof of Theorem \ref{second-main-result}]
    Let $(u,v)$ be solution of \eqref{e-1-radial} obtained in Theorem \ref{first-main-result}. We recall that $u$ has the following variational characterization
    \begin{equation}\label{critical-characterization}
    I(u)=\inf_{\gamma\in \Gamma}\sup_{t\in [0,1]}I(\gamma(t))=:c
\end{equation}
where $\Gamma=\{\gamma\in C([0,1],V): \gamma(0)=0\neq\gamma(1),~I(\gamma(1))\leq0\}$.

Now, suppose by contradiction that $(u,v)$ is radial. Let $(\lambda_1,\phi)$ be the first eigenpair of
    \begin{equation}
        \left\{
        \begin{aligned}
            -\phi''(r)-\frac{(N-1)\phi'(r)}{r}+\frac{2N\phi(r)}{r^2}+\phi(r)&=\lambda_1v(r)^{\frac{p-2}{2}}u(r)^{\frac{q-2}{2}}\phi(r),~~~~r\in(R,\infty)\\
            \phi(r)&=0,~~~~~~~~~~~~~~~~~~~~~~~~~~~~~~~r=R.
        \end{aligned}
        \right.
    \end{equation}
Consider also the first eigenpair $(\mu_1,\eta)$ of the eigenvalue problem
\begin{equation}\label{i}
    \left\{
    \begin{aligned}
        -\eta''(\theta)+\eta'(\theta)\kappa(\theta)&=\mu_1\eta(\theta)~~~~\text{in}~~(0,\frac{\pi}{2})\\
        \eta&>0~~~~~~~~~~~\text{in}~~(0,\frac{\pi}{2})\\
        \eta'(0)=\eta'(\frac{\pi}{2})&=0
    \end{aligned}
    \right.
\end{equation}
where $\kappa(\theta)=(m-1)\tan(\theta)-\frac{(n-1)}{\tan(\theta)}$. It is known that $\mu_1=2N$ and $\eta(\theta)=\frac{m-n}{N}-\cos(2\theta)$.
    
We now set $w(x)=w(r,\theta)=\phi(r)\eta(\theta)$. Then,
    \begin{align}
     \nonumber   -\Delta w(x)&+w(x)=-w_{rr}-\frac{(N-1)w_r}{r}-\frac{w_{\theta\theta}}{r^2}+\frac{w_{\theta}}{r^2}\kappa(\theta)+w(r,\theta)\\
     \nonumber  &=-\phi_{rr}(r)\eta(\theta)-\frac{(N-1)\phi_r(r)\eta(\theta)}{r}-\frac{\phi(r)\eta''(\theta)}{r^2}+\frac{\phi(r)\eta'(\theta)}{r^2}\kappa(\theta)+\phi(r)\eta(\theta)\\
     \nonumber  &=-\phi_{rr}(r)\eta(\theta)-\frac{(N-1)\phi_r(r)\eta(\theta)}{r}+\frac{2N\phi(r)\eta(\theta)}{r^2}+\phi(r)\eta(\theta)\\
        &=\lambda_1v(r)^{\frac{p-2}{2}}u(r)^{\frac{q-2}{2}}\phi(r)\eta(\theta)=\lambda_1v(|x|)^{\frac{p-2}{2}}u(|x|)^{\frac{q-2}{2}}w(x). 
    \end{align}
Now, define $\gamma_{\tau}(t)=t(u+\tau w)k$, where $k>0$ is chosen in such a way that $I((u+\tau w)k)\leq0$ for all $|\tau|\leq1$. By definition, $\gamma_{\tau}\in\Gamma$. Notice that there exists a unique smooth real function $f$ on a small neighborhood of zero with $f'(0)=0$ and $f(0)=\frac{1}{k}$ such that $\max_{t\in[0,1]}I(\gamma_{\tau}(t))=I(f(\tau)(u+\tau w)k)$.  

Let $g:\R\to\R$ be given by
	\begin{equation*}
	g(\tau)=I(f(\tau)(u+\tau w)k)-I(u). 
	\end{equation*}
	Then $g(0)=0$. Moreover a simple calculation yields $g'(0)=0$ since $I'(u)=0$. 
 
 We claim that if \eqref{condition-for-symmetry-breaking} holds true, then
 \begin{equation}
     g''(0)<0. 
 \end{equation}
 Indeed,
 \begin{align*}
     g''(0)&=I''(u)(w,w)\\
     &=(p'-1)\int_{A_R}|-\Delta u+u|^{p'-2}(-\Delta w+w)^2\ dx-(q-1)\int_{A_R}|u|^{q-2} w^2\ dx\\ 
     &=(p'-1)\lambda_1^2\int_{A_R}(v(|x|)^{p-1})^{p'-2}v(|x|)^{p-2}u(|x|)^{q-2} w(x)^2\ dx-(q-1)\int_{A_R}|u|^{q-2}w^2 \\
     &=(p'-1)\lambda_1^2\int_{A_R}u(|x|)^{q-2} w(x)^2\ dx-(q-1)\int_{A_R}|u|^{q-2}w^2\ dx\\
     &=((p'-1)\lambda_1^2-(q-1))\int_{A_R}u(|x|)^{q-2} w(x)^2\ dx.
 \end{align*}
 Thus
 \begin{equation*}
     g''(0)<0~~\text{if and only if}~~\lambda_1^2<\frac{q-1}{p'-1}=(p-1)(q-1).
 \end{equation*}
 Now, by definition,
 \begin{align*}
     \lambda_1&=\inf_{0\neq\phi\in H^1_0(A_R)}\frac{\int_{A_R}|\nabla\phi|^2\ dx+2N\int_{A_R}\frac{\phi^2}{|x|^2}\ dx+\int_{A_R}|\phi|^2\ dx}{\int_{A_R}\phi^2v(x)^{\frac{p-2}{2}}u(x)^{\frac{q-2}{2}}\ dx}\\
     &\leq \inf_{0\neq\phi\in H^1_0(A_R)}\frac{\Big(1+\frac{2N}{\Lambda_H}\Big)\int_{A_R}|\nabla\phi|^2\ dx+\int_{A_R}|\phi|^2\ dx}{\int_{A_R}\phi^2v(x)^{\frac{p-2}{2}}u(x)^{\frac{q-2}{2}}\ dx}\\
     &\leq \inf_{0\neq\phi\in H^1_0(A_R)}\frac{\Big(1+\frac{2N}{\Lambda_H}\Big)\Big(\int_{A_R}|\nabla\phi|^2\ dx+\int_{A_R}|\phi|^2\ dx\Big)}{\int_{A_R}\phi^2v(x)^{\frac{p-2}{2}}u(x)^{\frac{q-2}{2}}\ dx}\\
     &\leq\sqrt{\frac{q}{p}}\Big(1+\frac{2N}{\Lambda_H}\Big).
 \end{align*}
 In the latter, we have used Theorem \ref{t}. So, if
 \begin{equation}\label{condition-break-symmetry}
     \frac{q}{p}\Big(1+\frac{2N}{\Lambda_H}\Big)^2<(p-1)(q-1),
 \end{equation}
then
\begin{equation*}
    \lambda_1^2<(p-1)(q-1)
\end{equation*}
and thus
\begin{equation*}
    g''(0)<0. 
\end{equation*}
This implies that, assuming \eqref{condition-break-symmetry}, 
	\begin{equation*}
	\max_{t\in[0,1]}I(\gamma_{\tau}(t))=I(f(\tau)(u+\tau w)k)<I(u)
	\end{equation*}
	for small $\tau$. So, there exists $\tau>0$ small enough such that
	\begin{equation*}
	c\leq\max_{t\in[0,1]}I(\gamma_{\tau}(t))<I(u).
	\end{equation*}
	This contradicts \eqref{critical-characterization}.
\end{proof}
To prove Theorem \ref{third-main-result}, we first establish the following lemma which shows that solutions are distinct for different decompositions of $\R^N$. 
\begin{lemma}\label{distinct-solutions}
	Let $1< n<n'\leq\Big\lfloor\frac{N}{2}\Big\rfloor$ and let $m=N-n$ and $m'=N-n'$. Let $(u_{m,n},v_{m,n})$ resp. $(u_{m',n'},v_{m',n'})$ be the solution obtained Theorem \ref{first-main-result} corresponding to the decomposition $\R^m\times\R^n$ resp. $\R^{m'}\times\R^{n'}$ of $\R^N$. Then $(u_{m,n},v_{m,n})\neq (u_{m',n'}, v_{m',n'})$ unless they are both radial functions.  
\end{lemma}

\begin{proof}
	The proof of this lemma takes inspiration from \cite[Lemma 4.1]{cowan2022supercritical}. Assume that $(u_{m,n},v_{m,n})=(u_{m',n'},v_{m',n'})$ and let us prove that there are both radial. Since $u_{m,n}=u_{m',n'}$ then $u_{m,n}(s,t)=u_{m',n'}(s',t')$ where
	\begin{equation*}
	s=\sqrt{x_1^2+\cdots+x_m^2},~~~~~~ t=\sqrt{x_{m+1}^2+\cdots+x_N^2},
	\end{equation*}
	and
	\begin{equation*}
	s'=\sqrt{x_1^2+\cdots+x_{m'}^2},~~~~~~ t'=\sqrt{x_{m'+1}^2+\cdots+x_N^2}.
	\end{equation*}
	In particular, assuming $x_i=0$ for all $i\neq 1, m'$ we get that
	\begin{equation*}
	u_{m,n}(|x_1|,|x_{m'}|)=u_{m',n'}(\sqrt{x_1^2+x_{m'}^2},0) 
	\end{equation*}
	and thus, $u_{m,n}$ must be radial. On the other hand, assuming $x_i=0$ for all $i\neq m+1, N$, we have
 \begin{equation*}
     u_{m',n'}(|x_{m+1}|,|x_N|)=u_{m,n}(0,\sqrt{x_{m+1}^2+x_N^2}).
 \end{equation*}
 So, $u_{m',n'}$ is radial. Thus, $u_{m,n}$ and $u_{m',n'}$ are both radial functions. In the same manner, we also show that $v_{m,n}$ and $v_{m',n'}$ are radial. This completes the proof.
\end{proof}
We now give the proof of Theorem \ref{third-main-result}.
\begin{proof}[Proof of Theorem \ref{third-main-result}]
    By Theorem \ref{first-main-result}, it follows that for any $n\leq k$ and $q\geq p>2$, problem \eqref{e-1-radial} has a positive solution $(u_{m,n},v_{m,n})=(u_{m,n}(s,t),v_{m,n}(s,t))$ where
    \begin{equation*}
        s=\sqrt{x_1^2+\cdots+x_m^2},\quad\quad\quad t=\sqrt{x_{m+1}^2+\cdots+x_N^2} 
    \end{equation*}
    provided that
    \begin{equation}
        \frac{1}{p}+\frac{1}{q}>1-\frac{2}{n+1}~~~~\text{for}~~~n>\frac{p+1}{p-1}.
    \end{equation}
    Since $n\leq k$ then 
    \begin{equation*}
        1-\frac{2}{n+1}\leq 1-\frac{2}{k+1}~~~~\text{for}~~~k>\frac{p+1}{p-1}.
    \end{equation*}
    So, for every $n\leq k$, problem \eqref{e-1-radial} has positive solution provided that
    \begin{equation}
        \frac{1}{p}+\frac{1}{q}>1-\frac{2}{k+1}~~~~\text{for}~~~k>\frac{p+1}{p-1}.
    \end{equation}
    On the other hand, if $k\leq\frac{p+1}{p-1}$ then $n\leq\frac{p+1}{p-1}$ and thus \eqref{e-1-radial} has positive solution $(u_{m,n},v_{m,n})$ thanks to Remark \ref{rmk1}.

    Now, by Theorem \ref{second-main-result} the solution $(u_{m,n},v_{m,n})$ of \eqref{e-1-radial} obtained in Theorem \ref{first-main-result} is non-radial if
    \begin{equation*}
        (p-1)(q-1)>\Big(1+\frac{2N}{\Lambda_H}\Big)^2\Big(\frac{q}{p}\Big).
    \end{equation*}
    Consequently, for all $n\in\{2,\dots,k\}$ the solution $(u_{m,n},v_{m,n})$ is non-radial. On the other hand, by Lemma \ref{distinct-solutions}, $(u_{m,n},v_{m,n})\neq (u_{m',n'},v_{m',n'})$ for $n\neq n'$. This shows that we have $k-1$ distinct non-radial solutions. 
\end{proof}

 ~\\
\textbf{Acknowledgements:} The author is pleased to acknowledge the support of Fields Institute. He would like also to thank Abbas Moameni for some valuable discussion. The author finally thanks the anonymous referee who has carefully read the previous version of this paper.

\bibliographystyle{ieeetr}

\end{document}